\documentclass[11pt,reqno]{amsart}

\usepackage{amsmath,amsthm,amssymb,comment,fullpage}
\usepackage{braket}
\usepackage{mathtools}

\usepackage{mathrsfs}
\usepackage{caption}
\usepackage{centernot}
\usepackage{times}
\usepackage[T1]{fontenc}
\usepackage{blindtext}
\usepackage{hyperref}
\usepackage{latexsym}
\usepackage[dvips]{graphics}
\usepackage{epsfig}
\usepackage{svg}
\usepackage{amsmath,amsfonts,amsthm,amssymb,amscd}
\input amssym.def
\input amssym.tex
\usepackage{color}
\usepackage{hyperref}
\usepackage{url}
\usepackage{cases}
\newcommand{\bburl}[1]{\textcolor{blue}{\url{#1}}}

\usepackage{tikz}
\usepackage{tkz-tab}
\usetikzlibrary{shapes.geometric,positioning}

\newcommand{\burl}[1]{\textcolor{blue}{\url{#1}}}

\newcommand{\abs}[1]{\left|#1\right|}

\newcommand{\floor}[1]{\left\lfloor#1\right\rfloor}

\numberwithin{equation}{section}

\newtheorem{thm}{Theorem}[section]
\newtheorem{conj}[thm]{Conjecture}

\newtheorem{prop}[thm]{Proposition}

\newtheorem{defn}[thm]{Definition}

\theoremstyle{plain}

\newtheorem{corollary}[thm]{Corollary}

\newtheorem{lemma}[thm]{Lemma}

\newtheorem{theorem}[thm]{Theorem}
\newtheorem{conjecture}[thm]{Conjecture}

\newtheorem{remark}[thm]{Remark}

\newcommand\be{\begin{equation}}
\newcommand\ee{\end{equation}}
\newcommand\bee{\begin{equation*}}
\newcommand\eee{\end{equation*}}
\newcommand\bea{\begin{eqnarray}}
\newcommand\eea{\end{eqnarray}}
\newcommand\beae{\begin{eqnarray*}}
\newcommand\eeae{\end{eqnarray*}}
\newcommand\bi{\begin{itemize}}
\newcommand\ei{\end{itemize}}
\newcommand\ben{\begin{enumerate}}
\newcommand\een{\end{enumerate}}
\newcommand\bc{\begin{center}}
\newcommand\ec{\end{center}}
\newcommand\ba{\begin{array}}
\newcommand\ea{\end{array}}

\newcommand{\R}{\ensuremath{\mathbb{R}}}

\newcommand\frakfamily{\usefont{U}{yfrak}{m}{n}}
\DeclareTextFontCommand{\textfrak}{\frakfamily}

\newcommand{\hr}[1]{\href{#1}{\url{#1}}}

\newcommand{\E}[1]{\mathbb{E}[#1]}

\title{Distinct Angle Problems and Variants}

\author{Henry L. Fleischmann}
\email{\textcolor{blue}{\href{mailto:henryfl@umich.edu}{henryfl@school.edu}}}
\address{Department of Mathematics, University of Michigan, Ann Arbor, 48109}

\author{Hongyi B. Hu}
\email{\textcolor{blue}{\href{mailto:hongyih@andrew.cmu.edu}{hongyih@andrew.cmu.edu}}}
\address{Department of Mathematical Sciences, Carnegie Mellon University, Pittsburgh, 15213}

\author{Faye Jackson}
\email{\textcolor{blue}{\href{mailto:alephnil@umich.edu}{alephnil@umich.edu}}}
\address{Department of Mathematics, University of Michigan, Ann Arbor, 48109}

\author{Steven J. Miller}
\email{\textcolor{blue}{\href{mailto:sjm1@williams.edu}{sjm1@williams.edu}},  \textcolor{blue}{\href{Steven.Miller.MC.96@aya.yale.edu}{Steven.Miller.MC.96@aya.yale.edu}}}
\address{Department of Mathematics and Statistics, Williams College, Williamstown, MA 01267}

\author{Eyvindur A. Palsson}
\email{\textcolor{blue}{\href{mailto:palsson@vt.edu}{palsson@vt.edu}}}
\address{Department of Mathematics, Virginia Tech, Blacksburg, VA 24061}

\author{Ethan Pesikoff}
\email{\textcolor{blue}{\href{mailto:ethan.pesikoff@yale.edu}{ethan.pesikoff@yale.edu}}}
\address{Department of Mathematics, Yale University, New Haven, CT 06511}

\author{Charles Wolf}
\email{\textcolor{blue}{\href{mailto:charles.wolf@rochester.edu}{charles.wolf@rochester.edu}}}
\address{Department of Mathematics, Rochester, NY, 14627}

\thanks{This work was supported by NSF grant 1947438 and Williams College. E. A. Palsson was supported in part by Simons Foundation grant $\#$360560.}



\date{\today}

\begin{document}

\maketitle

\begin{abstract} 
The Erd\H{o}s distinct distance problem is a ubiquitous problem in discrete geometry. Less well known is Erd\H{o}s' distinct angle problem, the problem of finding the minimum number of distinct angles between $n$ non-collinear points in the plane. The standard problem is already well understood. However, it admits many of the same variants as the distinct distance problem, many of which are unstudied.

We provide upper and lower bounds on a broad class of distinct angle problems. We show that the number of distinct angles formed by $n$ points in general position is $O(n^{\log_2(7)})$, providing the first non-trivial bound for this quantity. We introduce a new class of asymptotically optimal point configurations with no four cocircular points. Then, we analyze the sensitivity of asymptotically optimal point sets to perturbation, yielding a much broader class of asymptotically optimal configurations. In higher dimensions we show that a variant of Lenz's construction admits fewer distinct angles than the optimal configurations in two dimensions.

We also show that the minimum size of a maximal subset of $n$ points in general position admitting only unique angles is $\Omega(n^{1/5})$ and $O(n^{\log_2(7)/3})$. We also provide bounds on the partite variants of the standard distinct angle problem.
\end{abstract}

\tableofcontents


\section{Introduction}
\subsection{Background} In 1946, Erdős published a paper titled ``On sets of distances of $n$ point", introducing the problem of finding asymptotic bounds on the minimum number of distinct distances among sets of $n$ points in the plane \cite{ErOg}. This simply stated problem proved to be surprisingly challenging and is now known as the Erdős distance problem. Indeed, the original question was only finally resolved by Guth and Katz in 2015 \cite{GuthKatz}.

Over time, many variations of the problem were introduced: restricting the point sets, studying subsets with no repeated distances, and many other quantities. We study variations of a related problem, introduced by Erdős and Purdy \cite{ErPur}. What is $A(n)$, the minimum number of distinct angles formed by $n$ not all collinear points on the plane? Unlike in the distance setting, an extra restriction of non-collinearity is required to prevent the degenerate case of at most two angles. When this problem was proposed, the regular $n$-gon was conjectured to be optimal (yielding $n-2$ angles), and a lower bound of $(n-2)/2$ angles was proven for point sets without three collinear points. Since then, the problem and all other analogues of distinct distance problems with angles have gone untouched. We study this problem of distinct angles in many of the settings originally considered for distinct distances, providing exact or asymptotic bounds, depending on the problem. We summarize our results below.


\subsection{Summary of Results and Methods} \label{subs: Summary of Results}
Note that throughout, unlike Erd\H{o}s, we do not count angles of $0$ or $\pi$ to avoid some degenerate behaviors. This is consistent with the current literature on related repeated angle questions (see, for example, \cite{PaSha}).
\subsubsection{Erdős Angle Problems}
We begin with the most natural extension of the Erd\H{o}s distance problems to angles: what is the least number of distinct angles determined by $n$ not all collinear points in the plane? Given that the known low angle constructions contain obvious structures, such as many points on a line or on a circle, it is natural to also consider the problem over restricted point sets.
Our main results in this section are summarized in the following.
\begin{itemize}
    \item We provide a construction of a polygon projected onto a line, yielding a point configuration with no four points on a circle admitting $n-2$ distinct angles, the same as the conjectured optimal regular $n$-gon.
    \item We also provide a point configuration in general position admitting less than $cn^{\log_2(7)}$ for some constant $c$, the first nontrivial bound for this problem. The configuration relies on enumerating classes of triangle equivalent up to edge translation and then projecting onto a generic plane. While \cite{EHP} and \cite{EFPA} also use a projection onto a generic plane for the distance problem analogue, the properties of orthogonal projections are much more convenient for distances. Attempting to directly apply their results fails in a dramatic fashion due to additional complexity added by angles.
\end{itemize}

In addition, for completeness, we provide a known upper bound on $A(n)$ of $n-2$ from the regular $n$-gon, conjectured by Erdős and Purdy to be the optimal configuration, and a lower bound of $n/6$ by partial progress towards the Weak Dirac Conjecture of Erd\H{o}s and Dirac. We also consider similar problems on restricted point sets, as in the distance setting. Under a restriction forbidding three points on a line, we provide bounds on the restricted quantity $A_{\text{no3l}}(n)$ (these were known by Erd\H{o}s).   

\subsubsection{The Robustness of Efficient Point Configurations}
Having identified efficient point configurations in the polygon and the projected polygon, we ask how resilient they are to perturbation.  Erd\H{o}s investigated a similar question for distances in \cite{ErRobust}.  We prove combinatorially the surprising result that both the polygon defines\footnote{We write $f = O(g)$ if there exists constant $C$ such that $f(n) \leq C g(n)$ for sufficient large $n$. Conversely, we write $f = \Omega(g)$ if $g = O(f)$. Lastly, we write $f = \Theta(g)$ if $f = O(g)$ and $f = \Omega(g)$.}
$O(nk)$ angles with $k$ points perturbed, so long as they all remain on the circle, and we prove an analogous result for the projected polygon. Consequently, if any constant number of points in a regular polygon are moved to random positions on the circle, the construction still defines $O(n)$ angles (an optimal number asymptotically), even though moving even one point off the circle experimentally gives a super-linear number of distinct angles.  In that vein, we provide conjectures about the number of angles in several perturbed optimal configurations in which points may no longer lie exactly on a circle or line.

\subsubsection{The Pinned Angle Problem}
We subsequently examine the angle equivalent of a prominent Erdős distance problem: namely, given $n$ points, what is the minimum number of distances determined between one ``pinned'' point and the rest, in the worst case?  This problem remains open in the distance setting for convex configurations of points, and is conjectured to be $\floor{n/2}$ by Erdos in \cite{ErOg}. An upper bound of $\floor{n/2}$ is obtained by considering the regular $n$-gon, and the current best lower bound of $\left(13/36 + 1/22701\right)n + O(1)$ is obtained by Nivasch, Pach, Pinchasi, and Zerbib \cite{NiPaPiZe}. 

Denoting the analogous angular quantity allowing any configuration of points as $\hat{A}(n)$ (with the pinned point as the center-point of the angles), we bound it between $n/6$ and $n-2$ using related $A(n)$ proofs. This in turn also provides an upper bound on $\hat{A}_\Sigma(n)$, the sum of the number of distinct angles determined by each point.

\subsubsection{Partite Sets}
Given a partite set, the question of distances determined between the two sets has been studied by Elekes \cite{El} in the unrestricted setting, but remains unsolved in general. We ask the analogous question in the angular setting: how many angles are defined by a $k$-partite set, where each point is in a distinct set? We provide low angle configurations in the unrestricted case, establish linear lower and upper bounds on partite sets without three collinear points, and completely solve the problem in a particular case.

\subsubsection{Maximal Subsets of Points with Distinct Angles}
One prominent variation of the Erdős distance problem asks: what is minimum maximal subset of $n$ points such that no distance is repeated? None of the numerous variants in the distance setting have been fully resolved, although a number of upper and lower bounds have been proven by a variety of authors. For a complete picture of these problems in the distance setting see \cite{BMP, Cha,LeTh,ErGuy,ErPur}. 

We ask the analogous question for angles. We upper bound this configuration in general, showing $R(n) \leq A(n)^{1/3}$. Then, we employ a probabilistic method similar to that in \cite{Cha} to show a lower bound of $\Omega(n^{1/5})$. 

\subsubsection{Higher Dimensions: Lenz's Construction}
The construction consists of multiple unit circles, each in a disjoint pair of dimensions. We show that, just as it has been for repeated angle problems in higher dimensions and the unit distance problem, Lenz's construction also provides a good upper bound of $2\lceil 2n/d \rceil - 2$ on $A_d(n)$, the least number of distinct angles defined by $n$ points in $d$ dimensions (see page 499 of \cite{BMP}).  This construction demonstrates that, for a fixed number of points, increasing the dimension decreases the upper bound on the number of distinct angles dramatically. This behavior aligns with the behavior in the distance setting with the integer lattice. We also provide a higher dimensional upper bound for the maximal subset question. \\

\begin{table}[h!]
    \begin{tabular}{|c|c|c|}
    \hline
         \textbf{Variant} &  \textbf{Lower Bound} & \textbf{Upper Bound}\\
         \hline
         \hyperlink{A(n) link}{$A(n)$} & $n/6$ & $n-2$ \\
         \hline
        \hyperlink{A no 3 l link}{$A_{\text{no}3l}(n)$} & $(n-2)/2$ & $n-2$ \\
        \hline
        \hyperlink{A no 4 c link}{$A_{\text{no4c}}(n) $} & $n/6$ & $n-2$ \\
        \hline
        \hyperlink{A gen}{$A_{\text{gen}}(n)$} & $\Omega(n)$ & $O(n^{\log_2(7)})$ \\
        \hline
        \hyperlink{A hat}{$\hat{A}(n)$} & $n/6$ & $n-2$ \\
        \hline
        \hyperlink{A Sigma}{$\hat{A}_{\Sigma}(n)$} &  $n/6 + n-1$ & $3n-6$ \\
        \hline
        \hyperlink{R gen}{$R_{\text{gen}}(n)$} & $\Omega(n^{1/5})$ & $O(n^{\log_2(7)/3})$ \\
        \hline
        \hyperlink{A d}{$A_d(n)$} & $2$ & $2\left \lceil 2n/d  \right\rceil -2$\\ 
        \hline
    \end{tabular}
    \caption{Summary of Results.}
    \label{tab:table of results}
\end{table}
We provide a tabular summary of our results in this section for convenience. Each parameter is described informally in Table \ref{tab:table of results} but formally defined in its respective section.

We note several other miscellaneous bounds over the course of the paper, but do not include them here because we only provide an upper or lower bound or they do not fit nicely into the structure of the table. These include: variants of angle sum bounds in Section \ref{sec: pinned}, partite set bounds in Section \ref{sec: partite sets}, and maximal subset bounds in non-general position in Section \ref{sec: max subsets}, to name a few.

\section{Erdős Angle Problems}

\subsection{Unrestricted Point Sets}
We begin by considering the most broad, non-trivial version of the distinct angles problem.  This is the version of the problem originally posed by Erd\H{o}s \cite{ErPur}. 

\begin{defn} \label{defn: min distinct angles}
For a point set $\mathcal P \subset \R^2$, let $A(\mathcal P)$ denote the number of distinct angles in $(0,\pi)$ determined by points in $\mathcal P$. Then, let 
\[
    \hypertarget{A(n) link}{A(n) \coloneqq \min_{|\mathcal P| = n} A(\mathcal P)},
\]
where $\mathcal P$ is not all collinear.
\end{defn}

We begin by showing $A(n) = \Theta(n)$ with explicit upper and lower bounds.  First, we give an upper bound using the regular polygon:

\begin{lemma} \label{lem: reg poly angles}
$A(n) \leq n-2$.
\end{lemma}
\begin{proof}
    Consider the point configuration given by the vertices of an $n$-sided regular polygon. Upon inscribing the polygon in a circle, notice that distinct angles are in bijection to the arclengths on the circle. We may fix a point as the central point of our angles by the symmetry of the polygon. There are then exactly $n-2$ possible arc lengths subtending angles with this central point.
\end{proof}

\begin{remark}\label{rmk: even polygon add point}
    When $n$ is odd, we may alternatively use an $(n-1)$-gon with an extra point in the center.  Adding the center point to an even regular polygon does not increase the number of nonzero angles defined, and so, if $n = 2m+1$, we achieve a slightly better bound: $A(2m + 1) \leq 2m-2$.
\end{remark}

We now use progress on the Weak Dirac Conjecture to provide a lower bound on $A(n)$. In 1961, based on a stronger conjecture of Dirac's, Erd\H{o}s conjectured in \cite{ErConj} the following.
\begin{conj} [Weak Dirac Conjecture] \label{conj: weak dirac}
    Every set $\mathcal{P}$ of $n$ non-collinear points in the plane contains a point incident to at least $\left \lceil n/2 \right \rceil$ lines of $\mathcal{L(P)}$, where $\mathcal{L(P)}$ is the set of lines formed by points.
\end{conj}

While the Weak Dirac Conjecture is open, significant progress has been made. Let $\ell(n)$ be the largest proven lower bound proven for the Weak Dirac Conjecture, i.e., every set $\mathcal{P}$ of $n$ points not on a line in the plane contains a point incident to at least $\ell(n)$ lines of $\mathcal{L(P)}$. Then, we have the following.
\begin{theorem} \label{thm: lower bd distinct angles, no res}
For $n > 3$, $A(n) \geq \frac{\ell(n) - 1}{2} \geq \frac{n}{6}$.
\end{theorem}
\begin{proof}
Fix a set $\mathcal{P}$ of $n$ non-collinear points in the plane. Let $p \in \mathcal{P}$ be incident to at least $\ell(n)$ lines of $\mathcal{L(P)}$. Fix another point $q$. Note that for any fixed nonzero angle $\theta < \pi$, there are exactly two possible lines where $r$ must lie on for $\angle qpr = \theta$. Since $p$ is incident to $\ell(n) - 1$ lines without $q$, $p$ is the center angle of at least $(\ell(n) - 1)/2$ distinct angles. Therefore
\[
    A(n) \geq \frac{\ell(n) - 1}{2}.
\]
We have $\ell(n) \geq \left \lceil n/3\right \rceil +1$ from \cite{Ha}. As such, we have $A(n) \geq n/6$, as desired.
\end{proof}

Notably, this argument is known (see Conjecture 10 in 6.2 of \cite{BMP}), but is included for completeness.

\subsection{No Three Collinear Points}
Given that any collinear point set defines at most two angles, it is intuitively clear why restricting the number of collinear points might result in interesting behavior. We briefly consider such point sets in this section.
\begin{defn}
Let \hypertarget{A no 3 l link}{$A_{\text{no}3l}(n) = \min_{|\mathcal P|=n} A(\mathcal P)$}, where $\mathcal P$ contains no collinear triples.
\end{defn}

Note that the regular $n$-gon contains no collinear triples, and so as with $A(n)$, we have an upper bound of $A_{\text{no}3l}(n)\leq n-2$. The usual stipulation on this bound holds. See Remark \ref{rmk: even polygon add point}. Our restrictions on the point set allow for a stronger lower bound. This bound was known by Erd\H{o}s but is included for completeness.

\begin{lemma} \label{lem: col3angles}
For $n >3$, $A_{\text{no}3l}(n) \geq \frac{n-2}{2} $.
\end{lemma}
\begin{proof}
Fix a point $p \in \mathcal P$. As no three points are on a line, $p$ determines $n - 1$ distinct lines with each of the other points. The result follows by fixing another point $q$ and repeating the argument for Theorem \ref{thm: lower bd distinct angles, no res}.
\end{proof}

We can easily generalize this restriction to no $k$ points on a line. However, the lower bound given by repeating this argument with $k \geq 4$ points on a line is always weaker than that in Theorem \ref{thm: lower bd distinct angles, no res}. Moreover, in those cases the regular polygon remains an upper bound.

\subsection{Restricting Cocircularity} 
Since the regular polygon  construction requires many points on a circle, it is natural to wonder how the bounds change when we require that no four points lie on a circle. This setting is not specifically studied in the context of distances but merits special attention for angles given the seeming optimality of the regular $n$-gon. We provide the following definition.
\begin{defn} \label{defn: no 4 co-circ}
Let \hypertarget{A no 4 c link}{$A_{\text{no4c}}(n) = \min_{|\mathcal P|=n} A(\mathcal P)$}, where $\mathcal P$ contains no co-circular quadruples. 
\end{defn}

We then have the following lemma.
\begin{lemma} \label{lem: one point off line}
For $n > 3$,
\[
A_{\text{no4c}}(n) \leq n - 2.
\]

\end{lemma}

\begin{proof}
    Consider the vertices of a regular $n$-gon. Fix a vertex $p$. Then, if $n$ is even, there is a vertex $q$ directly opposite $p$. In that case, let $\ell$ be the line perpendicular to $\overline{pq}$. If $n$ is odd, there are instead two vertices of minimal distance from $p$. In that case, let $\ell$ instead be the line those two vertices. Then, for each vertex $r$ other than $p$ in the regular $n$-gon, project $r$ onto $\ell$ at the intersection of $\overline{pr}$ and $\ell$. This is the stereographic projection of the points onto $\ell$ via $p$. Let the $n-1$ projected points on $\ell$ and $p$ define the projected polygon configuration, $\mathcal{P}$.
    
    Note that $\mathcal{P}$ contains no four cocircular points (Figure \ref{project-poly-to-line}). 
    
\begin{figure}[h!]
        \centering
 
\includegraphics[width=.6\columnwidth]{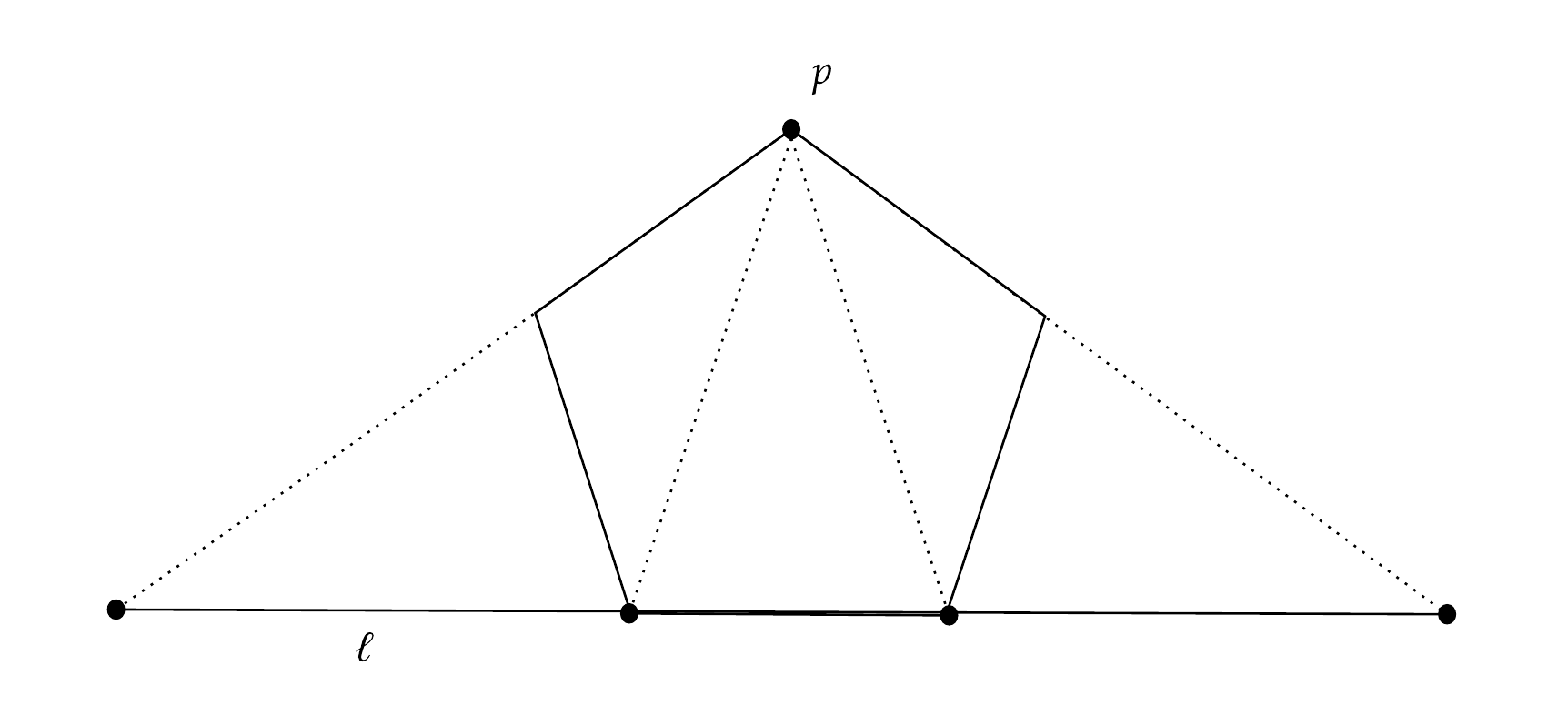}
\includegraphics[width=.6\columnwidth]{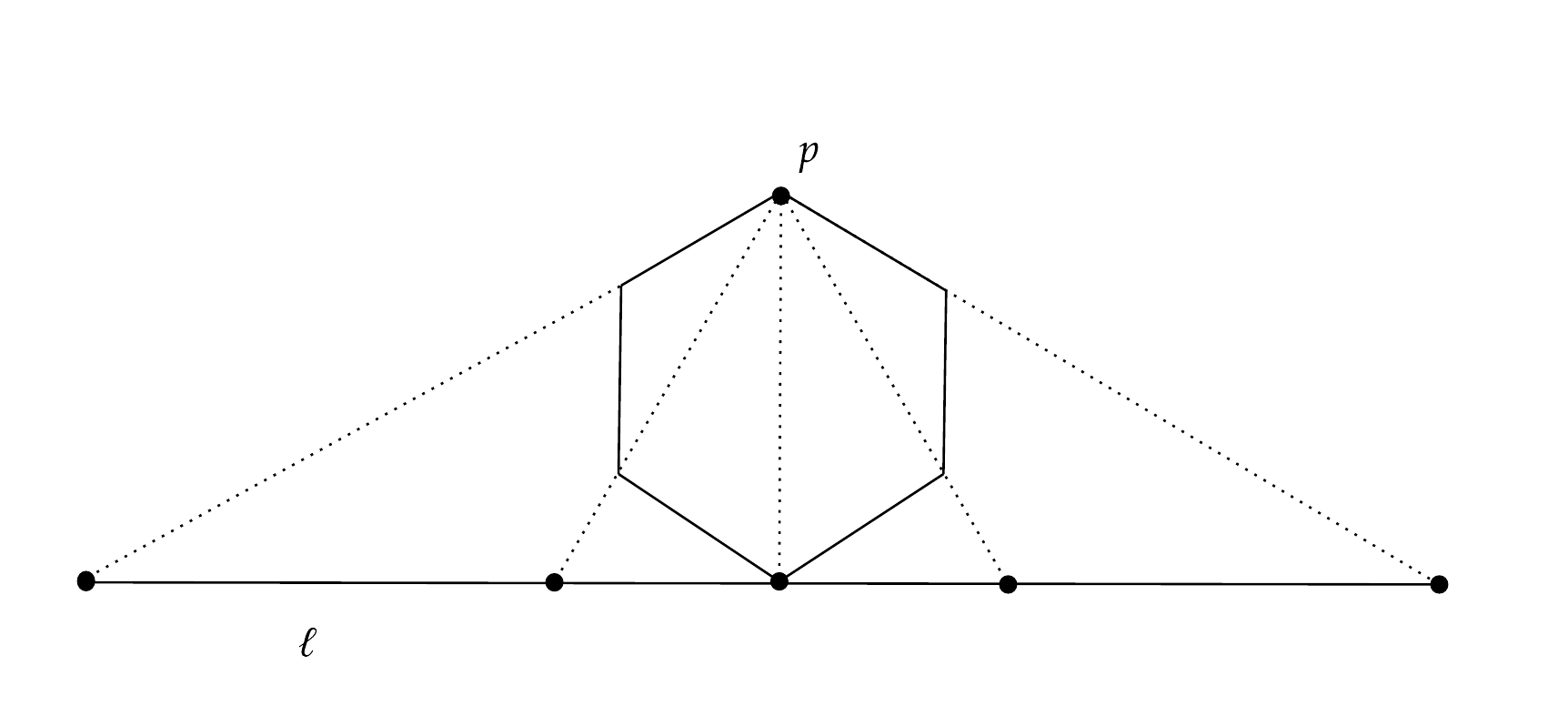}

        \caption{Projecting Regular Polygons onto a Line.}
        \label{project-poly-to-line}
    \end{figure}

We can now count the number of angles in this configuration. Let $\alpha = \pi/n$, the angle subtended by an arc between consecutive points in a regular $n$-gon. 

Note that the angles formed in the case of $p$ being the center of the angle are exactly the $n-2$ angles of a regular $n$-gon, $i\alpha = i\pi / n$ for $1\leq i\leq n-2$.

Next, we count angles of the form $\angle p q_i q_j$ where $q_i$ and $q_j$ lie on line $\ell$. We do not consider when all three points are on $\ell$ as that forms degenerate angles. We may assume that $q_i$ and $q_j$ are both on the same half of the line by the reflectional symmetry of the configuration.  We consider two cases: $i>j$ or $i<j$.  

Suppose first that $n$ is even. Then there will be a point $q_0$ at the center of the line.  First, we count angles with $i>j$. Notice that $\angle p q_i q_j = \angle p q_i q_0$.  The other two angles in $\triangle p q_i q_0$ are $i\alpha$ and $\pi/2$, so
\[
    \angle p q_i q_j = \frac{\pi}{2}-\frac{i\pi}{n} = (n/2 - i)\alpha, 
\]
for $1 \leq i \leq (n-2)/2$. Since $n$ is even, these are integer multiples of $\alpha$. Moreover, since $1 \leq n/2-i \leq n-2$, angles of this form are already accounted for in the angles with center $p$. Next, we examine the case of $i<j$.  Then $0 \leq i \leq n/2 - 1$.  Except for the angle with center at $q_0$, which has value $\pi/2$ (accounted for already in case 1), all these angles are the supplements of angles $i>j$.  That is, $\angle pq_i q_j = \pi - \angle p q_i q_0$.  Thus, we achieve angles of 
\[
    \pi-\frac{(n/2 - i)\pi}{n} = \frac{\pi}{2} + \frac{\pi i}{n},
\] 
where $1 \leq i\leq (n-2)/2$. All these angles are also accounted for in the case of angles with $p$ as the center, so, when $n$ is even, we have $n-2$ distinct angles.

Now suppose $n$ is odd. In this case, there will not be point $q_0$ opposite $p$ and on $\ell$, but we introduce one that is not in the configuration for later convenience.  As before, we first count angles with $i>j$.  There is a special angle that we will add to this count, namely $\angle p q_1 q_{-1}$.  Then, in effect, we are considering angles $\angle p q_i q_0$ for $i>0$ as in the odd case.  Each of these angles is in the triangle $\triangle p q_0 q_i$, whose other angles are $\pi/2$ and $i\alpha - \alpha/2$, for $1 \leq i \leq (n-1)/2$. Then \[
    \angle p q_i q_j = \frac{\pi}{2}-\frac{i\pi}{n} + \frac{\pi}{2n} = \frac{(\left \lceil n/2 \right \rceil - i)\pi}{n}.
\]
This is an integer multiple of $\pi/n$ , and, further, $1 \leq \left \lceil n/2 \right \rceil -i \leq n-2$.  Thus, all these angles are already accounted for in case 1.  Next, we examine the case of $i<j$.  Then $1 \leq i \leq (n-3)/2$. All these angles are the supplements of angles $i>j$.  That is, $\angle ap_i p_j = \pi - \angle a p_i p_0$.  Thus, we achieve angles of 
\[
    \pi-\frac{(\left \lceil n/2 \right \rceil - i)\pi}{n} = \frac{\pi}{2} + \frac{\pi i}{n}
\] 
where $1 \leq i \leq (n-1)/2$.  All these angles are also accounted for in case 1, and so when $n$ is odd, we have $n-2$ distinct angles.
     
Therefore, this configuration determines exactly $n-2$ angles. 
\end{proof}

\begin{remark} \label{rmk: projection}
The projected polygon construction and the regular polygon both give $n-2$ angles for $n$ points. The former contains collinear points but no four points on a circle, while the latter contains the opposite.  

Additionally, there are infinitely many such ``one point off the line" configurations yielding $\leq cn$ angles, for some $c$. Fix $\alpha < \pi/(n-1)$. Fix some $p$ and some line $\ell$. Space the remaining $n-1$ points on $\ell$ such that $\angle rps = \alpha$ for consecutive $r$ and $s$ on $\ell$. This configuration forms at most $3n$ angles in general. We revisit this in Section \ref{sec: robustness section}.
\end{remark}

Note that Theorem \ref{thm: lower bd distinct angles, no res} provides a lower bound of $n/6$ distinct angles here as well.

\subsection{General Position: No three points on a line nor four on a circle}

Now that we have illustrated constructions determining $O(n)$ angles that forbid either three points on a line or four points points on a circle, we consider configurations that forbid both. Erd\H{o}s and others have investigated this problem extensively in the distance setting. While the best known lower bound in the distance setting is trivially $\Omega(n)$, the best known upper bound is $n2^{O(\sqrt{\log n})}$ from \cite{EFPA}. In this section we provide a nontrivial upper bound on this quantity.

\begin{defn} \label{defn: gen config distinct angles}
Let 
\[
\hypertarget{A gen}{A_{\text{gen}}(n) \coloneqq \min_{|\mathcal P|=n} A(\mathcal P)},
\]
where $\mathcal P$ is in general position.
\end{defn}

We use a construction inspired by the projective construction in Theorem 1 of \cite{EHP} to provide an upper bound. We take higher dimensional hypercubes and project their points down to a generic plane. Unlike with distances, we have very little control over the triangles in the projection. As such, we proceed by very careful combinatorics.

Let $Q_d$ be the $d$-dimensional hypercube with vertices of the form $p = (x_1, x_2, \ldots, x_d)$ and $x_i = 0$ or $1$ for each $i$. For any 2-dimensional plane $\Pi$ in $\mathbb{R}^d$, let $T$ be the orthogonal projection of the points in $Q_d$ onto $\Pi$. It is possible to choose $\Pi$ satisfying the followingl conditions.
\begin{enumerate}
    \item $T(p_1) = T(p_2)$ if and only if $p_1 = p_2$, and
    \item $\mathcal{P} \coloneqq T(Q_d) \subset \Pi$ is in general position. 
\end{enumerate}
In addition, since orthogonal projections are self-adjoint and idempotent, we have 
\begin{equation}\label{eqtn: projection-property}
    p_1 - p_2 = p_3 - p_4 \implies d(T(p_1), T(p_2)) = d(T(p_3), T(p_4))
\end{equation}
by the following computation.
\begin{align*}
    \left<T(p_1-p_2), T(p_1-p_2) \right> &= \left<p_1 - p_2, T(T(p_1 - p_2)) \right> \\
    &= \left<p_3 - p_4, T(p_1 - p_2) \right> \\
    &= \left<T(p_3 - p_4), p_1 - p_2 \right>  \\
    &= \left<T(T(p_3 - p_4)), p_3 - p_4 \right> \\
    &= \left<T(p_3 - p_4), T(p_3 - p_4) \right>.
\end{align*}
Unfortunately, $T$ does not preserve the distance between points in $Q_d$. That is,
\[
    d(p_1, p_2) = d(p_3, p_4) \centernot \implies d(T(p_1), T(p_2)) = d(T(p_3), T(p_4)).
\]
This means two congruent triangles in $Q_d$ need not be congruent after projection. However, by (\ref{eqtn: projection-property}) and SSS-congruence, two congruent triangles with equal difference vector edges remain congruent after projection. This inspires the following definition.
\begin{defn} \label{defn: triangle congruence}
    Given a triangle $\Delta$ with vertices in $Q_dd$, define the equivalence class $[\Delta]_{Q_d}$ as the set of all triangles congruent to $\Delta$ whose vertices lie in $Q_d$ and edges correspond to (individually) translated copies of the edges of $\Delta$.
\end{defn}

It suffices to characterize the equivalence classes of translated congruent triangles in $Q_d$ to bound the number of angles in $P$. We do so in the following lemma.
\begin{lemma} \label{lem: triangle classes Q_d}
    The number of equivalence classes of triangles in $Q_d$ is 
    \[
        \frac{7^d -3^{d+1} + 2}{12}.
    \]
\end{lemma}
\begin{proof}
We begin by counting the number of unordered triples of distinct binary $k$-tuples such that no coordinate of the triple is fixed for all three. By ``fixed," we mean equal among all three $k$-tuples. Let $a_k$ denote the number of such triples. 



At each coordinate in the triple of $k$-tuples, the possible values for each of the triples are $0$ or $1$. Since no coordinate of the triples is fixed, each coordinate must either have exactly one 1 or one 0 among the three $k$-tuples. There are then $(3 \cdot 2)^k$ ways to choose whether there will be one 1 or one 0 and the choice of tuple for that singleton for each of the $k$ coordinates. This imposes an ordering which we divide out at the end. Now, note that although it is impossible to repeat a $k$-tuple thrice since no coordinate is fixed, we overcount instances with a repeated $k$-tuple. A $k$-tuple can be chosen to repeat twice in $2^k$ ways and the choice of repeated $k$-tuple forces the choice of the third $k$-tuple. Such triples can be ordered in 3 ways. After subtracting off such pairs, the remaining triples are all distinct and thus can be ordered in $3!$ ways. We have
\[
    a_k = \frac{6^{k} - 3 \cdot 2^k}{6}.
\]

Now we use $a_k$ to count $t_k$, the number of triangles in $Q_d$ with exactly $k$ unfixed coordinates. As there are ${ d \choose k}$ ways to choose the unfixed coordinates, $a_k$ ways to choose the values in those unfixed coordinates, and $2^{d-k}$ ways to choose the values of the fixed coordinates, we get
\[
t_k = \frac{6^{k} - 3 \cdot 2^k}{6}{ d \choose k}2^{d-k} = { d \choose k}\frac{6^{k}2^{d-k} - 3\cdot 2^d}{6}.
\]

Finally, we prove that all triangles in the same equivalence class have the same number of fixed coordinates. We also prove that the size of the equivalence class of triangles with a given number of fixed coordinates is constant.

First observe that, given a triangle in $Q_d$, we may get an equivalent triangle by flipping any combination of its fixed coordinates (from $0$ to $1$ or vice versa), thereby translating each vertex of the triangle by the same amount. Moreover, if a coordinate is \textit{not} fixed, then any translation to a equivalent triangle cannot be nonzero in that coordinate as it would then lead to some point having a coordinate that is both nonzero and not one. The only other way to achieve an equivalent triangle is to translate the edges individually (keeping the difference vectors corresponding to them identical, but altering their relative orientations). This can happen in at most one way.

Moreover, flipping each coordinate of each of the three points yields a congruent triangle composed of the same difference vectors (by flipping, we mean changing from $0$ to $1$ and vice versa). This follows from the observation that, if a coordinate in a difference vector is 0, the subtracted coordinates must be equal and remain equal under swapping each coordinate. If a coordinate is $1$ or $-1$, it will swap sign but still be the same vector.

Thus, we have that the size of the equivalence class of each triangle with $k$ unfixed points in a $Q_d$ is exactly $2^{d-k + 1}$. Moreover, all triangles in the equivalence class have exactly $k$ unfixed points.

Putting this all together, we have that the number of equivalence classes of triangles in $Q_d$ is
\begin{align*}
    \sum_{k = 1}^{d} 
    \frac{{ d \choose k}(6^{k}2^{d-k} - 3\cdot 2^d)}{6 \cdot 2^{d-k+1}} &= \sum_{k = 1}^{d} 
    { d \choose k} (6^k/12 - 2^{k-2}) \\
    &= \frac{1}{12} \left[ \sum_{k = 1}^{d} 
    { d \choose k}6^{k} - 3\sum_{k = 1}^{d} 
    { d \choose k}2^{k} \right] \\
    &= \frac{7^d}{12} - \frac{1}{12} - \frac{3^{d}}{4} + \frac{1}{4} \\
    &= \frac{7^d -3^{d+1} + 2}{12}.
\end{align*}
The above follows from standard binomial formula identities. 
\end{proof}

Since triangles in the same equivalence class are congruent under the projection from $Q_d$ to a specially chosen generic plane, Lemma \ref{lem: triangle classes Q_d} provides an upper bound of $O(7^d)$ distinct angles on point configurations in general position with $2^d$ points. To establish this result for $n$ not a power of two, pick the least $d$ such that $n < 2^d$ and apply the upper bound to a subset of $n$ vertices in $Q_d$. This proves the following theorem.
\begin{theorem} \label{thm: distinct angles gen config}
$A_{\text{gen}}(n) = O(n^{\log_2(7)})$.
\end{theorem}

\section{The Robustness of Efficient Point Configurations} \label{sec: robustness section}
Before we consider several variants of this problem, we discuss a very natural question: how far can point sets stray from our best constructions (regular polygons and projected polygons) while still being ``near-optimal?'' In the distance setting, a point set is called near-optimal if it admits $O(n/\sqrt{\log n})$ angles, like the $\sqrt{n} \times \sqrt{n}$ integer lattice. Erd\H{o}s asked if such sets have lattice-like structure, containing $\Omega(n^{1/2})$ points on a line. The question has gone unsolved even after replacing $1/2$ by any $\epsilon > 0$. For some partial results related to this problem, see \cite{ShZaZe},  \cite{RaRoSh}, and \cite{PaZe}, in which the authors bound the number of points on various algebraic curves in near-optimal point sets.

Formally, we have the following definition.
\begin{defn} \label{defn: near optimal}
Let $\mathcal{P}_1$, $\mathcal{P}_2$, $\ldots$ $\subset \mathbb{R}^2$ be a sequence of point configurations with $| \mathcal{P}_n| = n$. Then $\mathcal{P}_n$ is \textit{near-optimal} if $A(\mathcal{P}_n) = O(n)$.
\end{defn}
For example, the sequence of point configurations of regular $n$-gons is near-optimal.  Perhaps surprisingly, these configurations are reasonably robust to point perturbation. We begin by studying points on a circle in a way reminiscent of a regular $n$-gon.

\begin{prop}
For a fixed $k \geq 0$, let $\mathcal{S}_n^k$ be the collection of $n$ points on a circle with $n-k$ points forming a regular $(n-k)$-gon and the remaining $k$ placed arbitrarily. Then
\[
\max_{\mathcal{P} \in \mathcal{S}_n^k}A(\mathcal{P}) = \Theta(nk)
\]
\end{prop}

\begin{proof}
    Fix a configuration $\mathcal P \in S_n^k$. Since points in $\mathcal P$ lie on a circle, all its angles are incident angles. Thus the number of distinct angles is bounded by the number of distinct arc lengths. We divide the set of arcs into three cases.  Suppose an arc is the minor one formed between $p,q \in \mathcal P$. We then have the following cases.
    \begin{enumerate}
        \item Both $p,q$ are on the polygon. There are at most $n-k-2$ distinct arc lengths of this form.
        \item Neither of $p,q$ are on the polygon. Then, they are among the $k$ arbitrarily placed points. There are at most $2\binom{k}{2} = k^2-k$ distinct arc lengths of this form.
        \item We have $p$ is on the polygon and $q$ is not. There are at most $2(n-k)k = 2nk-2k^2$ distinct arc lengths of this form.
    \end{enumerate}
    Then, in total, $\mathcal P$ determines at most $2nk - k^2 + n -2k - 2$ angles. Since $k\leq n$, this quantity is $O(nk)$.
    
    We show that this bound is tight. Choose the $k$ points to be placed at multiples of $2\pi$ between $\frac{n-k-1}{n-k} \cdot 2 \pi$ and $2\pi$ such that all arcs formed with the added point are iteratively not admitted among the points in the configuration.  Then there are $n-k - 2$ arcs between the point at  $\frac{n-k-1}{n-k} \cdot 2 \pi$ and the other non-zero polygonal points, moving clockwise about the circle. Notably,  they can be extended by $1$ to $k$ arcs. By the choice of these arcs, for any length rational arc chosen, all the extensions by $1$ to $k$  arcs are distinct. Thus, at least $(n-k - 2)k = \Omega(nk)$ distinct angles are formed (as the ends of the arc can be used as the end points of an angle with 0 as the center).
\end{proof}

From this, it follows that such configurations are near-optimal for any $k$ constant in $n$.

We now discuss a related, more challenging problem. Let $\mathcal{T}_n^k$ be the collection of $n$ point configurations having $n-k$ points on a circle and no circle containing more. Denote by $T(n)$ the maximum quantity $k \leq n/2$ satisfying
\[
    \min_{\mathcal{P} \in \mathcal{T}_n^k} A(\mathcal{P}) = \Theta(n).
\]
What can be said about the value of $T(n)$? We restrict $k$ to $n - k = \Omega(n)$ in order to avoid cases that reduce to configurations with a negligible number of points on a circle.

Consider a new point at the center of a regular $n$-gon. When $n$ is even, the new point generates no new angles (see Remark \ref{rmk: even polygon add point}). When $n$ is odd, it generates exactly $\lceil n/2 \rceil$ additional angles. To see this, say $n = 2m+1$. One new angle of $\frac{4m\pi}{2m+1}$ is from the angle whose center is at the new point and the ends at the original $n$-gon. Another $n$ more angles of $\frac{i\pi}{n} - \frac{\pi}{2n}$ for $1 \leq i \leq n$ are from those with the new point as an end point; they are equivalent to the original arcs on the $n$-gon with half of the $2\pi/n$ arc cut out.

Thus, $T(n) \geq 1$. We conjecture this bound to be tight and this point in the center is the only way to achieve it. 
\begin{conjecture} \label{conj: more than 1 point off a circle}
    T(n) = 1.
\end{conjecture}

The aforementioned optimal point configuration, the projected polygon, has $n-1$ points be on a line. To what extent can we perturb this configuration while remaining near-optimal?

\begin{prop}
For a fixed $k \geq 0$, let $\mathcal{L}_n^k$ be the collection of planar $n$-point configurations with  with $n-k$ (including the off-line point) in the projected polygon construction from Lemma \ref{lem: one point off line} and the remaining $k$ placed on the line arbitrarily. Then
    \[
    \max_{\mathcal{P} \in \mathcal{L}_n^k} = A(\mathcal{P}) = \Theta(nk).
    \]
\end{prop}

\begin{proof}
    We may divide the possible angles into four cases. Observe an angle $\angle pqr$ in the following cases.
    \begin{enumerate}
        \item Say $q$ is the point off the line, and $p,r$ are both projected points.  Then the number of such angles is bounded by the number of angles in the projected polygon construction $n-k-2$.
        \item Say $q$ is the point off the line, and neither of $p,r$ are projected points.  Then the number of angles in this case is bounded by $\binom{k}{2} = (k^2-k)/2$.
        \item Say $q$ is the point off the line, $p$ is a projected point, and $r$ is not (without loss of generality).  Then the number of such angles is this bounded by $(n-k-1)k$.
        \item Say $q$ is on the line.  Then all nontrivial $\angle pqr$ have one leg along the line and the other at the point off the line.  Thus, for each $q$ on the line, there are at most two possible angles, and they are supplementary amongst themselves. Then the number of these angles is bounded by $2(n-1)$.
    \end{enumerate}
    Then in total we $O(nk)$ distinct angles formed.
    
    Now we show this bound is tight. Choose the $k$ points to be placed following the rightmost point on the line such that the angles they form with the point off the line as the center are all unique (we can do this iteratively by the infinitude of the real line). Then, by construction, there are at least
    \[
    \sum_{i = 1}^{k} (n-k + i - 1) = k(n-k) + {k \choose 2} = \Omega(nk)
    \]
    distinct angles in the configuration, counting only angles centered at the point off the line.
\end{proof}

From this, if follows that such configurations are near-optimal for any $k$ constant in $n$.

Now, we consider the quantity analogous to $T(n)$. Let $\mathcal{M}_n^k$ be the collection of all $n$ point configurations having $n-k$ points on a line and no lines containing more. Denote by $M(n)$ the maximum $k$ satisfying
\[
    \min_{\mathcal{P} \in \mathcal{M}_n^k} A(\mathcal{P}) = \Theta(n).
\]
What can be said about the value of $M(n)$?

It is immediate that $M(n) \geq 1$ by the projected polygon. In addition, from Remark \ref{rmk: projection}, there are many other near-optimal configurations in $\mathcal{M}_n^k$. Consider a new point that is the original point off the line reflected over the line. It generates $O(n)$ additional angles.

Thus, $M(n) \geq 2$. We conjecture this bound to be tight and the only pair of points to achieve this are those symmetric about the line.

\begin{conjecture} \label{conj: more than 2 points off a line}
    M(n) = 2.
\end{conjecture}

\section{The Pinned Angle Problem} \label{sec: pinned}
We now pivot to a variant originally considered in the context of distinct distance problems. Among $n$ points in the plane, in the worst case, what is the maximum number of distinct angles centered at some ``pinned'' point? For example, see Corollary 6 in \cite{KaTa} for the original problem for distances. 
\begin{defn}
For a point set $\mathcal P$ and a point $p$ in it, let $A_p(\mathcal P)$ denote the number of distinct angles in $(0,\pi)$ formed in $\mathcal P$ with $p$ as the center. Then, let
\[
\hypertarget{A hat}{\hat{A}(n)} \coloneqq \min_{|\mathcal{P}| = n} \max_{p \in P}A_p(\mathcal{P}),
\]
where $\mathcal{P}$ is a not all collinear planar point set.
\end{defn}
\begin{theorem} \label{thm: A hat}
We have $\hat{A}(n) = \Theta(n).$
\end{theorem}
\begin{proof}
Regular polygons give an upper bound of $n-2$. Let $\ell(n)$ be the current best lower bound on the Weak Dirac Conjecture (see Conjecture \ref{conj: weak dirac}). Using the same logic as Theorem \ref{thm: lower bd distinct angles, no res}, this gives a lower bound of $(\ell(n) - 1)/2$. This is at least $n/6$, completing the proof.
\end{proof}
This use of the Weak Dirac Conjecture is known (see Conjecture 10, Section 6.2 in \cite{BMP}).
A related classical question for distinct distances is determining the average number of distinct distances admitted by a pinned point. We present its analogy for angles here.
\begin{defn}
Let 
\[
\hat{A}_{\Sigma}(n) \coloneqq \min_{|\mathcal{P}| = n} \sum_{p \in P}A_p(\mathcal{P}),
\]
where $\mathcal{P}$ is a not all collinear planar point set.
\end{defn}

\begin{theorem} \label{thm: A hat sigma}
We have $\hat{A}_{\Sigma}(n) \leq 3n-6$.
\end{theorem}
\begin{proof}
The construction from Lemma \ref{lem: one point off line} has one point off the line as the center of $n-2$ distinct angles. All the points on the line, apart from the endpoints, are the center of exactly two non-degenerate angles. These contribute another $2n-4$ to the sum. Hence, $\hat{A}_{\Sigma}(n) \leq 3n-6$.
\end{proof}

We also have the following for $\hat{A}'_{\Sigma}(n)$, this quantity forbidding three collinear points.

\begin{theorem} \label{A hat sigma prime}
We have \hypertarget{A Sigma}{$\hat{A}'_{\Sigma}(n)$} $= \Theta(n^2)$.
\end{theorem}
\begin{proof}
Since no three points are on a line, we can repeatedly remove points and apply the bound from Theorem \ref{thm: lower bd distinct angles, no res}. This gives
\[
\sum_{i = c}^n \frac{i}{6} = \Omega(n^2).
\]
To get the upper bound, use the fact that every point of a regular $n$-gon is the center of exactly $n-2$ distinct angles.
\end{proof}

\section{Partite Sets} \label{sec: partite sets}
Another well known variant of the distinct distances problem is that of distances between points in partite sets. See \cite{El}, for example. We introduce a similar problem for angles. We make heavy use of the best lower bound of the Weak Dirac Conjecture on $n$ points, $\ell(n)$, in this section. As a reminder, from  \cite{Ha}, we have $\ell(n) \geq \left \lceil n/3 \right \rceil$.

\begin{defn} \label{defn: bipAngle}
Given bipartite $\mathcal{P}, \mathcal{Q} \subset \R^2$, denote by $A(\mathcal{P}, \mathcal{Q})$ the number of distinct angles in $(0,\pi)$ whose central vertex lies in a different set than its two end points.
Where $\mathcal{P} \cup \mathcal{Q}$ is not all collinear, let $$A(m,n) \coloneqq \min_{\abs{\mathcal{P}}=m, \abs{\mathcal{Q}}=n}
A(\mathcal{P}, \mathcal{Q}).$$
\end{defn}

For simplicity, assume $m \leq n$. We begin by providing upper bounds on both unrestricted and restricted point sets.

\begin{lemma} \label{lem: A(m,n) gen upper bound}
We have $A(m,n) \leq m$.
\end{lemma}
\begin{proof}
We utilize the projected polygon construction from Lemma \ref{lem: one point off line}. Assign the point off the line, $q$, to be in $\mathcal{Q}$ and the $m$ leftmost points on the line in $\mathcal{P}$. Notably, the points in $\mathcal{P}$ do not cross the center point of the points on the line.  Now, the angles of the form $\angle p_1 q p_2$, the angles with center in $\mathcal{Q}$, form the angles $\frac{i\pi}{n + m}$ for $1 \leq i \leq m - 1$. Let $r$ be the rightmost point on the line. From Lemma \ref{lem: one point off line}, we have two cases for the the angles of the form $\angle q p_i r$. If $n+m$ is even, they form the angles 
\[
\frac{\pi}{2} - \frac{j\pi}{n+m}
\]
for $\frac{n-m}{2} \leq j \leq \frac{n+m - 2}{2}$ (or, for $n = m$, $j \geq 1$ and these angles form an angle of $\pi/2$ for $p_i$ the orthogonal projection of $q$ onto the line). In the case $n+m$ odd, the angles $\angle q p_i r$ are 
\[
\frac{\pi}{2} - \left( \frac{j\pi}{n+m} - \frac{\pi}{2(n+m)} \right)
\]
for $\frac{n-m + 1}{2} \leq j \leq \frac{n+m-1}{2}$. Namely, these angles are computed by completing the angles of the right triangle containing $q$, $p_i$, and the orthogonal projection of $q$. Substituting the ranges of the angles for both, we find that the only angles in the configuration are
$\frac{l\pi}{n+m}$ for $1 \leq l \leq m$, implying the result.
\end{proof}
 Note that this result implies that no unrestricted lower bound in terms of $n + m$ can exist.

Next, we provide an upper bound on the case of no three collinear points in $\mathcal{P}$ or $\mathcal{Q}$.
\begin{lemma} \label{lem: BipGenUpperBound}
We have $A_{\text{no3l}}(m,n) \leq n - 2$.
\end{lemma}
\begin{proof}
Let $\mathcal{P}$ form a subset of the vertices of a regular $n$-gon of size $m$ and $\mathcal{Q}$ the vertices of a regular $n$-gon, both inscribed in the same circle. The bipartite angles formed in this configuration are all incident angles subtended by arcs of the size subtending angles in a regular $n$-gon. Thus, the distinct angles formed in this configuration are a subset of the angles of a regular $n$-gon, implying the result from Lemma \ref{lem: reg poly angles}.
\end{proof}

 Now, we provide a lower bound in the restricted case of no three collinear points within sets $\mathcal{P}$ or $\mathcal{Q}$.
\begin{lemma} \label{lem: bipLB}
We have $A_{\text{no3l}}(m,n) \geq \left \lfloor (n-1)/2 \right \rfloor$.
\end{lemma}
\begin{proof}
Fix a point $p \in \mathcal{P}$ to be a center point and a $q \in \mathcal{Q}$ to be a non-center point. By the argument in Lemma \ref{lem: col3angles}, since no three points are collinear, at most 2 of each of the remaining $n-1$ points in $\mathcal{Q}$ can form the same angle with $\overline{pq}$ with center $p$. However, one other point in the $\mathcal{Q}$ can be collinear to $p$ and $q$, not contributing any angle, yielding 
\[
\lceil (n-2)/ 2 \rceil = \lfloor (n-1)/2 \rfloor. \qedhere
\]
\end{proof}

\begin{corollary}
    We have $A_{\text{no3l}}(m,n) \geq \left \lfloor (\frac{m+n}{2}-1)/2 \right \rfloor$.
\end{corollary}

This allows us to completely solve this problem in the case of $m = 1$.
\begin{lemma}
We have $A_{\text{no3l}}(1, n) = \lfloor (n-1)/2 \rfloor$.
\end{lemma}
\begin{proof}
Let $\mathcal{Q}$ be the vertices of a regular $n$-gon. Inscribe these vertices in a circle. Let the singular point $p$ in the other set be the center of the circle. Then the number of angles of the form $\angle q_1pq_2$ for $q_1, q_2 \in \mathcal{Q}$ is $n/2 - 1$ for $n$ even and $(n-1)/2$ for $n$ odd by counting subtending arclengths. This yields $A(1,n) \geq \lfloor (n-1)/2 \rfloor$. Combining with Lemma \ref{lem: bipLB}, we achieve the desired result.
\end{proof}

We now consider $k$-partite sets.

\begin{defn} \label{defn: kPartAngle}
Let $n = \sum_{i=1}^{k} r_i$. Let $A(r_1,r_2, \dots, r_k)$ denote the minimum number of distinct angles determined by point sets in $\mathbb{R}^2$ of respective sizes $r_1, \dots, r_k$ with each of the following stipulations:
\begin{enumerate}
    \item each angle    is formed by three points in distinct sets,
    \item $r_1 \geq r_2 \geq \cdots \geq r_k$ and $k \geq 3$, and
    \item not all points are collinear.
\end{enumerate}
\end{defn}

We begin with upper bounds on unrestricted and restricted point sets.
\begin{lemma}
We have $A(r_1, r_2, \ldots, r_k) \leq 2(n - r_1 -1)$.
\end{lemma}
\begin{proof}
Let $S = n - r_1$. We follow a similar proof to Lemma \ref{lem: A(m,n) gen upper bound}. Let the $n$ points be in an $n$ point projected regular polygon configuration. Let the point off the line, $p$, be in the $r_1$ set. Let the leftmost $r_k$ points be in the $r_k$ set, the next leftmost in the $r_{k-1}$ set, and so on, with the remaining rightmost points on the line in the $r_1$ set. Note that all angles must include $p$.  The angles $i\pi/n$ for $1\leq i \leq S - 1$ are exactly those formed by angles with $p$ as the center. We assume without loss of generality that $S$ is at most $n/2$ as, from Lemma \ref{lem: one point off line}, there are $n-2$ total angles in the configuration. From the proof of Lemma \ref{lem: A(m,n) gen upper bound}, the angles centered at $p$ and the acute angles centered at some point in an $r_i$ set for $i > 1$ overlap completely, yielding $S-1$ angles (note that the rightmost center point has no endpoint for an angle opening to the right, thus the minus one). Moreover, since we may assume $S \leq n/2$, all of $S-1$ supplemental angles are obtuse. Thus, they do not overlap at all, yielding the desired bound of $2S-2$ angles.
\end{proof}
As before, we are unable to provide a lower bound only in terms of $n$, as, from the above, there are configurations with large $r_1$ which exhibit very few distinct angles.

We next provide an upper bound on $A(r_1, r_2, \ldots, r_k)$ in the restricted case of no three points on a line.
\begin{lemma}
We have $A_{\text{no3l}}(r_1,r_2, \ldots, r_k) \leq n-\max(2, r_k + 1)$.
\end{lemma}
\begin{proof}
Place the $n$ points as the vertices of a regular $n$-gon. Assign the first $r_1$ to the first partite set, the next $r_2$ to the second, and so on, continuing about the circle circumscribing the polygon clockwise. Then, since all points are on a circle, the angles in the configuration each correspond exactly to the arc subtending them. But, by construction, the arcs subtending angles may contain at most $n - \max(3, r_k + 2)$ points . As such, the distinct angles in this configuration are exactly $i\pi/n$ for $1 \leq i \leq n - \max(2, r_k + 1)$. This implies the desired bound.
\end{proof}

If we forbid three points on a line within each of constituent sets, we have a stronger bound than Lemma \ref{lem: bipLB} for $r_{k-1}$ and $r_k$ small.

\begin{lemma} \label{lem: kPartLowBound}
$A_{\text{no3l}}(r_1,r_2, \ldots r_k) \geq \lceil \frac{n-r_{k-1}- r_k}{2} \rceil$.
\end{lemma}
\begin{proof}
Fix a point in the $r_k$-set and a point in $r_{k-1}$-set. Now, since there are no three collinear points, the points in the other sets make each angle with these points at most twice. Thus, $A_{\text{no3l}}(r_1,r_2, \ldots r_k) \geq \lceil \frac{n-r_{k-1}- r_k}{2} \rceil$.
\end{proof}

\section{Maximal Subsets of Points with Distinct Angles} \label{sec: max subsets}
\subsection{Definitions and Upper Bounds}
Another variant of the Erd\H{o}s distinct distance problem is the following: given $n$ points in a plane, how many points must we remove so that the remaining points determine no repeated distances? This problem has been studied extensively in the context of distances, with varying restrictions on the points set \cite{Cha,GuthKatz,LeTh,ErGuy}. We study an analogous problem for angles, proving the first nontrivial lower and upper bounds.

\begin{defn} \label{defn: remove pts}
For a point set $\mathcal P$, let \hypertarget{R gen}{$R(\mathcal{P})$} be the maximum size of $\mathcal{Q}\subseteq\mathcal{P}$ such that $\mathcal Q$ determines no repeated angle. Then, let
\[
    R(n) \coloneqq \min_{|\mathcal{P}| = n} R({\mathcal{P}}),
\]
where the minimum is taken over non-collinear point sets $\mathcal P$ of $n$ points.
\end{defn}

In general, configurations with a low number of angles provide a reasonable upper bound for $R(n)$. 

\begin{lemma} \label{lem: bound on removed pts}
Let $\mathcal{P} \subseteq \mathbb{R}^2$ be a planar point configuration of $n$ points with no three collinear points. Then
\[
R(\mathcal{P}) \leq (2A(\mathcal{P}))^{1/3}.
\]
\end{lemma}
\begin{proof}
Fix $\mathcal{P} \subseteq \mathbb{R}^2$, a planar point configuration of $n$ points with no three collinear points. Then, the subsets of the point configuration determine at most $A(\mathcal{P})$ distinct angles. Moreover, as there are no three collinear points in $\mathcal{P}$, any subset $S$ of $\mathcal{P}$ admits $3{|S| \choose 3}$ not necessarily distinct angles. Thus, if a subset $S$ has no repeated angles, it must be that
\[
3{|S| \choose 3} \leq A(\mathcal{P}).
\]
This implies $|S| \leq (2A(\mathcal{P}))^{1/3}$.
\end{proof}

Using Lemma \ref{lem: reg poly angles} and Theorem \ref{thm: distinct angles gen config}, we get the following bounds.
\begin{corollary} \label{cor: bounds on R(n)}
We have
\begin{align*}
    R(n), R_{\text{no3l}}(n)  &= O(n^{1/3}). \\
    R_{\text{no4c}}(n),  R_{\text{gen}}(n) &=  O(n^{\log_2(7)/3}). 
\end{align*}
\end{corollary}

\begin{remark}
Notably, Lemma \ref{lem: bound on removed pts} does not provide an especially strong bound for $R_{\text{no4c}}(n)$, since the construction from Lemma \ref{lem: one point off line} has $n-1$ points on a line and removing the point off the line thus yields a subset with all (trivially) distinct angles.
\end{remark}

\subsection{A Probabilistic Lower Bound on Maximal Distinct Angle Subsets in general position}

Now we provide a lower bound on $R_{\text{gen}}(n)$. The proof is in many ways reminiscent of Charalambides' proof of this for distances (see Proposition 2.1 in \cite{Cha}). 

To proceed, we define and bound several quantities.

\begin{defn} \label{Q sets}
For a point set $\mathcal{P}$, let 
\begin{align*}
    \phantom{Q_4(\mathcal{P})}
    &Q_3(\mathcal{P}) \coloneqq \{ (p,q,r) \in \mathcal{P}^3 \,:\, p,q,r \text{  distinct, } \angle pqr = \angle qrp\},\\
    &\begin{aligned}
        Q_4(\mathcal{P}) \coloneqq \{ (p,q,r,s) \in \mathcal{P}^4 \,:\, p,q,r,s \text{ distinct}, \angle pqr &= \angle pqs, \angle pqr = \angle rsp, \\
    \angle pqr &= \angle qrs, 
    \text{ or } \angle pqs = \angle qrs \},
    \end{aligned} \\
    &\begin{aligned}
        Q_5(\mathcal{P}) \coloneqq \{ (p,q,r,s,t) \in \mathcal{P}^5 \,:\, p,q,r,s,t \text{ distinct}, \angle pqr &= \angle sqt, \\ \angle pqr &= \angle qst, \angle pqr = \angle rst\},
    \end{aligned} \\
    &Q_6(\mathcal{P}) \coloneqq \{(p,q,r,s,t,u) \,:\, p,q,r,s,t,u \text{ distinct } \angle pqr = \angle stu \}.
\end{align*}
\end{defn}
\begin{remark} \label{rmk: explanation of Q_i}
$Q_3(\mathcal{P})$ is the collection of pairs of equal angles overlapping at all three points. It is also the collection of isosceles triangles, over counting up to a factor of 3.

$Q_4(\mathcal{P})$ is the collection of pairs of equal angles overlapping at two points. The first case is when the angles share a central point and one endpoint. The second case is when they share both endpoints. The third case is when their center points are endpoints for the other angle (and since that gives two points overlapping, the other two endpoints do not overlap). The fourth and final case is when the endpoints of one angle are the center and an endpoint of the other. See Figure \ref{fig: q4}.

\begin{figure}[h!]
\centering
\includegraphics[scale = 0.9]{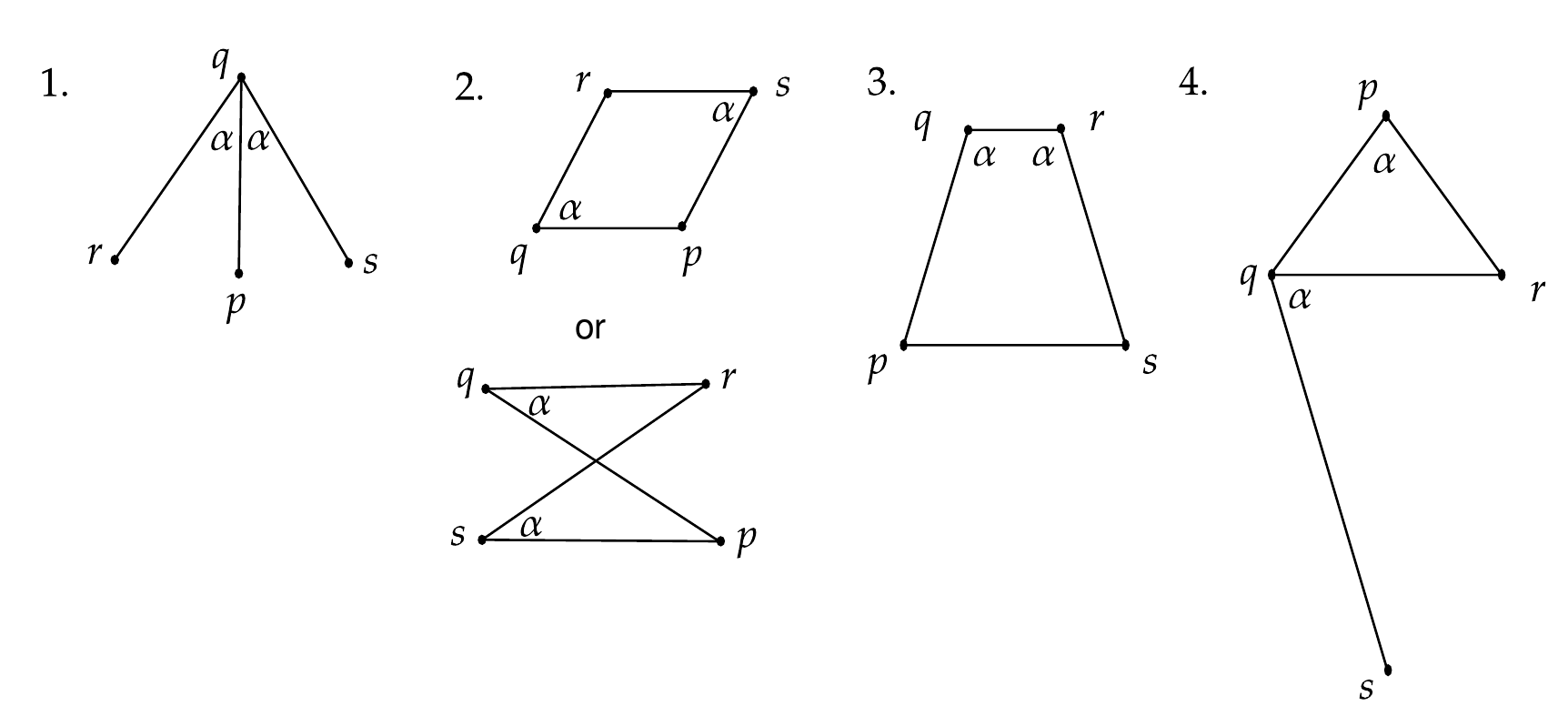}

\caption{Four Point Repeated Angle Configurations.}\label{fig: q4} \end{figure}

$Q_5(\mathcal{P})$ is the collection of pairs of equal angles overlapping at one point. The first case is when the angles share a central point. The second case is when the center of one angle is an endpoint of the other. The third case is when an endpoint of one angle is also an endpoint of the other. See Figure \ref{q5}.

\begin{figure}[h!]
\centering
\includegraphics[scale=0.9]{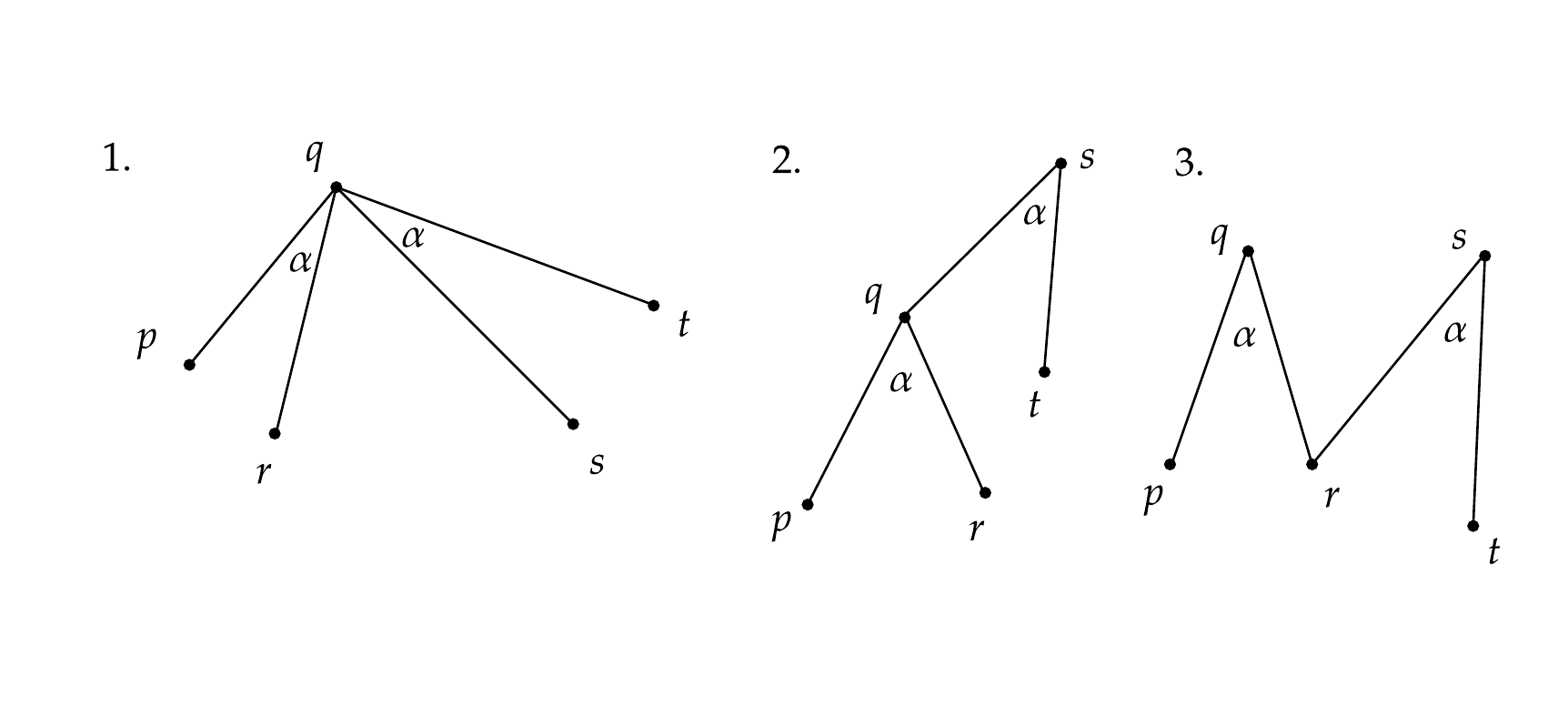}
\caption{Five Point Repeated Angle Configurations.}
\label{q5} \end{figure}

$Q_6(\mathcal{P})$ is the collection of pairs of equal angles without overlaps.

These cases are all encompassing as, for each number of repeated points, it involves some matching of angle endpoints/centers to angle endpoints/centers. These cases are an exhaustive list of such matchings.
\end{remark}

For that which follows, we assume the point configuration is in general position in the plane. That is, no three points in the configuration are on a line and no four are on a circle. The argument fails without restricting to sets with no four points on a circle, as remarked afterward. 

\begin{defn} \label{general Q}
For each $3 \leq i \leq 6$, let
\[
q_i(n) \coloneqq \max_{|\mathcal{P}|=n} |Q_i(\mathcal{P})|,
\]
where $\mathcal{P}$ is a planar point set in general position. 
\end{defn}

\begin{lemma} \label{Q3 bound}
We have $q_3(n) = O(n^{7/3}).$
\end{lemma}
Pach and Sharir show the number of isosceles triangles in the plane is $O(n^{7/3})$ \cite{PaSha}. We count such triangles twice in $Q_3(n)$ if the triple forms an isosceles triangle and thrice if it forms an equilateral triangle. This makes no difference asymptotically.

\begin{lemma} \label{Q4 bound}
We have $q_4(n) = O(n^3).$
\end{lemma}
\begin{proof}
Let $\mathcal{P}$ be $n$ points in general position. We show that there are at most $cn^3$ quadruples in each case of $Q_4(\mathcal{P})$, for some constant $c$, implying the desired bound. In each of the cases below, we use the same naming convention as in Definition \ref{Q sets}. Our case numbering also correspond to those in Figure \ref{fig: q4}.

Fix $p,q$ and $r$, which may be done in less than $n^3$ ways. 
\begin{description}
    \item[\textit{Case 1}] 
     We count the number of ways to choose $s$ such that $\angle pqs = \alpha = \angle pqr$. Since $\alpha$ is determined by $p,q$ and $r$, then $s$ must be on a ray from $q$ forming angle $\alpha$ with $\overline{qp}$. There are at most two such rays, with one being $\overrightarrow{qr}$. Since the points are in general position, $s$ cannot be on $\overrightarrow{qr}$. On the other ray, there are at most one such point $s$. So, there are at most $n^3$ quadruples in this case.
    \item[\textit{Case 2}] 
     Let $C$ be the circle determined by the three points and $C'$ be its reflection across the segment $\overleftrightarrow{pr}$. Let $\ell$ be the perpendicular bisector of segment $\overleftrightarrow{pr}$. We show that any point $s$ forming $\angle psr = \alpha$ has to lie on the outer perimeter of the figure $C \cup C'$. 
    
    First, assume point $s$ lies in the same half space as $q$ with respect to the line $\overleftrightarrow{pr}$ and the same half space as $r$ with respect to $\ell$. Let $s'$ be the point projected onto $C$ projected from $s$ along $\overleftrightarrow{ps}.$ By comparing the triangles $\triangle psr$ and $\triangle ps'r$, we have that $\angle ps'r = \angle psr$ iff $\angle prs = \angle prs'$. Therefore, it must be that $s = s'$. 
    
    The previous argument can be performed when $s$ is on the other side of $\ell$, by projecting along $\overleftrightarrow{rs}$ instead. However, in either of these cases, $s \in C$, which violates the restriction of no four points on a circle. That is, the second diagram in Case 2 of Figure \ref{fig: q4} is impossible given our restrictions.
    
    This argument can be repeated when $s$ is on the other side of $\overleftrightarrow{pr}$, but instead on circle $C'$. Since there are already two points on this circle, there is most one choice for $s$ in this case. Hence, there are at most $n^3$ quadruples in this case. 
    
    \item[\textit{Case 3}] 
     We choose $s$ such that $\angle qrs = \alpha$. Since $\alpha$ is determined by $p,q$ and $r$, $s$ must be on a ray from $r$ forming angle $\alpha$ with $\overline{qr}$. There are at most two such rays. Each such ray contains at most two point, yielding at most two options for $s$. So, there are at most $2n^3$ quadruples in this case.
    \item[\textit{Case 4}] 
    This case is extremely similar to Cases 1 and 3.    We choose $s$ such that $\angle sqr = \alpha = \angle pqr$. As before, there are at most two lines which intersect $\overline{qr}$ at $q$ at an angle of $\alpha$. Each line contains at most one additional point in $\mathcal{P}$, yielding an upper bound of $2n^3$ quadruplets in this case.
\end{description}
Since each four cases can occur in at most $cn^3$ ways, for $c = 2$, then $q_4(n) = O(n^3)$, as desired.
\end{proof}

\begin{remark} \label{rmk: necessity of gen. config for prob proof}
If four points on a circle is allowed, then we instead get $q_4(n) = \Theta(n^4)$, crippling the proof. This can be achieved with a regular $n$-gon. Fix two of its vertices as the overlapping endpoints. They partition the vertices of the polygon into major and minor arcs, with the major one having least $(n-2)/2$ vertices. Then, any choice of two vertices from the major arc will yield a configuration as in Case 2. Thus there are $\Omega(n^4)$ such configurations in this case. As such, forbidding four points on a circle is necessary. 
\end{remark}

\begin{lemma} \label{Q5 bound}
We have $q_5(n) = O(n^4).$
\end{lemma}
\begin{proof}
Let $\mathcal P$ be $n$ points in general position. As in Lemma \ref{Q4 bound}, we show that each case of $Q_5(\mathcal P)$ occurs at most $cn^4$ times, for constant $c$, implying the desired bound.
\begin{description}
    \item[\textit{Case 1}]
    Fix $p,q,r,$ and $s$, done in less than $n^4$ ways. As in Lemma \ref{Q4 bound}, there are at most two choices for $t$, yielding at most $2n^4$ options for this case.
    \item[\textit{Case 2}]
    This follows exactly as Case 1, yielding a bound of at most $2n^4$ options for this case.
    \item[\textit{Case 3}]
    This follows exactly as Cases 1 and 2, yielding a bound of at most $2n^4$ options for this case.
\end{description}
The above casework implies the result.
\end{proof}

\begin{lemma} \label{Q6 bound}
We have $q_6(n) = O(n^5).$
\end{lemma}
\begin{proof}
Fix $n$ points in general configuration and from it points $p,q,r,s,$ and $t$, in less than $n^5$ ways. There are then exactly two ways to choose $u$ so that $\angle pqr = \angle stu$. Then $u$ must lie on one of two lines containing $t$. Since there are no three collinear points, there are at most two choices of $u$. The result follows.
\end{proof}

\begin{theorem}
    We have $R_{\text{gen}}(n) = \Omega(n^{1/5})$.
\end{theorem}
\begin{proof}
    Let $\mathcal P \subset \R^2$ be a point set of size $n$ and let $\mathcal Q \subset \mathcal P$ be a set in which each element of $\mathcal Q$ is chosen independently and uniformly from $\mathcal{P}$ with probability $p$. The probability $p$ will be specified below.
    
    Each occurrence of some configuration from $\bigcup_{i=3}^6 Q_i$ in $\mathcal Q$ generates a repeated angle. Let $\mathcal Q' \subset \mathcal Q$ be the points remaining after one point from each configuration is removed. Indeed, $\mathcal Q'$ is free of repeated angles and $|\mathcal Q'| \leq R_{\text{gen}}(\mathcal P)$. 
    
    Taking expectations we obtain 
    \[\E{|\mathcal Q'|} \ge \E{|\mathcal Q|} - \sum_{i=3}^6 \E{|Q_i|} = pn - \sum_{i=3}^6 p^i q_i(n).\]
    Using Lemmas \ref{Q3 bound}--\ref{Q6 bound}, there exist some constant $c > 1$ such that for all $n > N$ for some $N$, we get 
    \[
    \E{|\mathcal Q'|} \geq np - c(p^3  n^{7/3} - p^4 n^3 - p^5 n^4 - p^6 n^5).
    \]
    Setting $p = c^{-1}n^{-4/5}$, for $n > N$ we have
    \begin{align*}
    \E{|\mathcal Q'|} &\geq c^{-1}n^{1/5} - c^{-2}n^{-1/15} - c^{-3}n^{-1/5} - c^{-4} - c^{-5}n^{1/5} \\
    &= \Omega(n^{1/5}).
    \end{align*}
    
    By the first moment method, there exists a subset of size $\Omega(n^{1/5})$ without repeated angles.
\end{proof}

\section{Higher Dimensions: Lenz's Construction} \label{sec: higher dim}
In general, the minimum number of distinct angles among $n$ points in $\mathbb{R}^d$ should decrease as lower dimensional spaces can be embedded in higher dimensional ones. In this section, we provide a construction that demonstrates that this is indeed the case. 

\begin{defn}
Let \hypertarget{A d}{$A_d(n)$} be the minimum number of distinct angles on three points determined by $n$ non-collinear points in $d$-dimensional space.
\end{defn}

In dimension $d$ for $d \geq 4$, Lenz gives a construction for a low upper bound on $A_d(2n)$, as described in \cite{ErLenz}. 

Construct a unit regular $n$-gon centered at the origin in the $x_1 x_2$-plane and another unit regular $n$-gon centered at the origin in the $x_3x_4$-plane. This is Lenz's construction. Now, we upper bound the number of distinct angles in this configuration. From Lemma \ref{lem: col3angles}, there are $n-2$ distinct angles between points in the same $n$-gon. There are then two other cases to consider. We may assume without loss of generality the points lie in four dimensions, as the extra dimensions make no difference in the computation.

Let the three points be
\begin{enumerate}
    \item $x = (\cos(\theta), \sin(\theta), 0, 0)$   
    \item $y = (\cos(\psi), \sin(\psi), 0, 0)$
    \item $z = (0,0, \sin(\phi), \cos(\phi))$,
\end{enumerate}
where $\theta, \psi, \phi \in \{2\pi i/n : 0 \leq i \leq n-1\}$ and $\theta \neq \psi$.

\begin{description}
\item[Case 1] The endpoints are in the same polygon.

In this case, $z$ is the center of the angle. We compute $\alpha = \angle xzy$ using
\begin{equation} \label{eqtn: law of cosines}
    \alpha \,=\, \arccos\left(\frac{\left< x-z, y - z\right>}{\left\Vert x-z \right\Vert \left\Vert y-z \right\Vert}\right).
\end{equation}
For the computations that follow, here is a useful trigonometric identity:
\begin{equation} \label{eqtn: cos id}
    \cos(\theta_1 - \theta_2) = \cos(\theta_1)\cos(\theta_2) + \sin(\theta_1)\sin(\theta_2).
\end{equation}
Now, for convenience, 
\begin{align*}
    x-z &= (\cos(\theta), \sin(\theta), -\sin(\phi), -\cos(\phi))\\
    y - z &= (\cos(\psi), \sin(\psi), -\sin(\phi), -\cos(\phi)).
\end{align*}
Substituting, we have
\begin{align*}
    \left< x-z, y - z\right> &= \cos(\theta)\cos(\psi) + \sin(\theta)\sin(\psi) + \sin^2(\phi) + \cos^2(\phi) \\
    &= 1 + \cos(\theta - \psi),
\end{align*}
where the second step follows from applying Equation \ref{eqtn: cos id}. Similarly, we have 
\begin{equation} \label{eqtn: norm vals}
    \left\Vert x-z \right\Vert = \left\Vert y-z \right\Vert = \sqrt{2}.
\end{equation}
Substituting into \eqref{eqtn: law of cosines}, we have 
\begin{equation*} 
    \alpha = \arccos \left({\frac{1+\cos(\theta - \psi)}{2}} \right).
\end{equation*}

We know $\theta - \psi = 2\pi k/n$ for nonzero $-n+1 \leq k \leq n-1$. Since cosine is an even and periodic with $2\pi$, the image of $\cos(\theta - \psi)$ are exactly $\cos(2\pi k/n)$ for $1 \le k \le \lceil \frac{n-1}{2} \rceil$. Because arccosine is injective on $[-1,1]$, it yields exactly $\lceil \frac{n-1}{2} \rceil$ values of $\alpha$.
\item[Case 2] The endpoints are in different polygons.

Here, we may assume $x$ is the center of the angle. We compute $\alpha = \angle yxz$ using
\begin{equation} \label{eqtn: law of cosines2}
    \alpha \,=\, \arccos\left(\frac{\left< y-x, z - x\right>}{\left\Vert y-x \right\Vert \left\Vert z-x \right\Vert}\right).
\end{equation}
Again, for convenience,
\begin{align*}
    y-x &= (\cos(\psi) - \cos(\theta), \sin(\psi) - \sin(\theta), 0 , 0) \\
    z-x &= (-\cos(\theta), -\sin(\theta), \sin(\phi), \cos(\phi)).
\end{align*}
Substituting, 
\begin{align*}
    \left< y-x, z - x\right> &= -\cos(\psi)\cos(\theta) + \cos^2(\theta) -\sin(\psi)\sin(\theta) + \sin^2(\theta) \\
    &= 1 - \cos(\theta - \psi),
\end{align*}
where the last step is another application of Equation \eqref{eqtn: cos id}.
As before, $\left\Vert z-x \right\Vert = \sqrt{2}$. However, we also have
\begin{align*}
    \left\Vert y-x \right\Vert &= \sqrt{\cos^2(\psi) + \cos^2(\theta) -2\cos(\psi)\cos(\theta) + \sin^2(\psi) + \sin^2(\theta) - 2\sin(\psi)\cos(\theta)} \\
    &= \sqrt{2 - 2\cos(\theta - \psi)} \\ 
    &= \sqrt{2}\sqrt{1 - \cos(\theta - \psi)}.
\end{align*}
Now, substituting into Equation \eqref{eqtn: law of cosines2}, we have
\begin{equation*}
    \alpha = \arccos \left(\frac{\sqrt{1 - \cos(\theta - \psi)}}{2} \right).
\end{equation*}
As in Case 1, we have $\left\lceil \frac{n-1}{2} \right\rceil$ possible values for $\cos(\theta - \psi)$.

Now, since in both cases arccosine is injective on the domain, these $\alpha$ are duplicate angles if and only if
\begin{equation*}
    1 + \cos(\theta - \psi) = \sqrt{1 - \cos(\theta - \psi)}.
\end{equation*}
Since both sides are non-negative we may square both sides. Thus, we have a duplicate angle if and only if we have a solution to 
\begin{equation*}
     \cos(\theta - \psi)(\cos(\theta - \psi) + 3) = 0.
\end{equation*}
So, $\cos(\theta - \psi)$ must equal zero. Since we are only considering $\theta - \psi = 2 \pi k/n$ for $1 \leq k \leq \left\lceil \frac{n-1}{2} \right\rceil$, this occurs only for $k = n/4$. Hence, we overcount between these two cases exactly once if and only if $4 \mid n$.
\end{description}

As a result of these computations, we have the following lemma.
\begin{lemma} \label{lem: Distinct Angles Lenz}
The number of distinct angles in Lenz's construction with $2n$ points is at most $2n - 4$ if $4 \mid n$ and at most $2n - 3$ otherwise.
\end{lemma}

We can extend Lenz's construction to get even better bounds on the minimum number of distinct angles in higher dimensions. In dimension $d \geq 6$, we may now have three unit regular $n$-gons in disjoint pairs of coordinates. Crucially, adding the third $n$-gon adds at most one angle, formed by points on three different polygons. As the distance between points is $\sqrt{2}$, three points always yield an equilateral triangle and an angle of $\pi/3$. 

If the dimension allows, you may then add even more $n$-gons in disjoint coordinates. After the third, the additional ones do not add any additional distinct angles. Since subsets of the vertices of $n$-gons have a subset of the angles, you can make point sets of any size using Lenz-like constructions.

\begin{theorem} \label{thm: LenzUpperBd}
Fix $d \geq 2$. For $n > d+1$, we have that
\begin{numcases}{2 \leq A_d(n) \leq}
    n-2, & $d = 2,3$ \label{p3} \\
    2\left \lceil \frac{n}{2} \right\rceil -3, & $d = 4,5$ \label{p2} \\
    2\left \lceil \frac{n}{\left\lfloor d/2 \right\rfloor} \right\rceil -2, & $d \geq 6$. \label{p1}
\end{numcases}
For $3 \leq n \leq d+1$, we have 
\begin{equation} \label{p4}
 A_d(n) = 1.  
\end{equation}


\end{theorem}
\begin{proof}
The lower bound of $2$ for \eqref{p3}, \eqref{p2}, \eqref{p1} follows from the fact that the $d$-dimensional simplex has exactly $d+1$ vertices and it is the largest point configuration in $d$-dimensional space with all points equidistant. As such, in each of these cases, there are at least two distances between the points and thus more than one distinct angle. 

Now note that $\left \lceil \frac{n}{\left\lfloor d/2 \right\rfloor} \right\rceil \geq 3$ if $n >  d + 1$. \eqref{p2} and \ref{p1} then follow from Lemma \ref{lem: Distinct Angles Lenz} and the above discussion of generalized Lenz's constructions. In fact, in special cases for $d \geq 6$ we get a slightly lower bound. For $ \frac{n}{\left\lfloor d/2 \right\rfloor}$ a multiple of $3$ or $4$, we may reduce the bound by 1 (and $2$ if a multiple of $12$).

For \eqref{p2} we also may reduce the bound by 1 if $\left \lceil \frac{n}{2} \right\rceil$ is a multiple of 4. 

For \eqref{p3}, for $d = 2$ or $3$, this follows from Lemma \ref{lem: col3angles} by using a regular $n$-gon.

For \eqref{p4}, note that for $n = d+1$, we can arrange the points to form the vertices of a $d$-dimensional regular simplex, yielding all points equidistant from one another. This means all angles are $\pi/3$ as they are the angle of an equilateral triangle. We can take any subset of the vertices of such a simplex to get the same result.
\end{proof}

Crucially, this construction implies that no uniform lower bound greater than 4 for a fixed $n$ and varying $d$ can exist.

From this, we can also give an upper bound on the higher dimensional version of the quantity $R(n)$ from Definition \ref{defn: remove pts}.

\begin{defn} \label{defn: remove pts d}
Let {$R_d(\mathcal{P})$} be the maximum size of any $\mathcal{Q}\subseteq\mathcal{P} \subseteq \mathbb{R}^d$ such that $\mathcal Q$ defines no angle twice.

Over the set of all $n$ non-collinear points, define 
\[
R_d(n) = \min_{|\mathcal{P}| = n} R_d({\mathcal{P}}). 
\]
\end{defn}

We again make use of a variation of Lenz's construction to provide an upper bound.

\begin{prop} \label{cor: removed lenz}
We have 
\[
R_d(n) \leq \left( 2\left \lceil \frac{n}{\lfloor d/2 \rfloor} \right \rceil - 4\right)^{\frac{1}{3}}.
\]
\end{prop}
\begin{proof}
We use the variation of Lenz's construction from Lemma \ref{lem: Distinct Angles Lenz}. Distribute the points as evenly as possible amongst the largest possible regular polygons in disjoint pairs of dimensions as normal. Note that there cannot be points on three different circles as they form an equilateral triangle. Also note that there cannot be two points on one circle and one on another as that forms an isosceles triangle. As such, we may apply Lemma \ref{lem: bound on removed pts} to the largest polygon to achieve our desired bound.
\end{proof}

\begin{remark} \label{rmk: 2-distance point configurations}
For $n \leq {d +1 \choose 2}$, there is a two distance set of that many points (see Lemma 3.1 of \cite{Li}). Thus, in such sets all but two points must be removed, yielding $R_d(n) = 2$ for $n \leq {d +1 \choose 2}$.
\end{remark}

\section{Future Work} \label{sec: future work}
Future research may take distinct angles problems in a number of new directions:
\begin{enumerate}
    \item We have shown that $n/6 \leq A(n) \leq n-2$. Further, we have identified two non-collinear point configurations which define exactly $n-2$ angles, the regular $n$-gon and its projection onto the line.  Whether these are in fact the optimal configurations is open (though they are conjectured to be so), and even if they are, there may be others which also define $n-2$ angles. Note that we have observed that, excluding angles of $0$ and $\pi$, one may add a point to the center of an even sided regular polygon without adding any angles. See Remark \ref{rmk: even polygon add point}. This does not contradict Erd\H{o}s' initial conjecture in \cite{ErPur}, as he included $0$ angles.
    \item Prove Conjecture \ref{conj: more than 1 point off a circle} and Conjecture \ref{conj: more than 2 points off a line} regarding optimal point configurations.
    \item One may similarly improve our bounds on $A_{\text{no}3\ell(n)}$, $A_{\text{no4c}}$, and $A_{\text{gen}}$.  In general position, an optimal construction has yet to be conjectured.
    \item The question of distinct angles in higher dimensional space has yet to be explored deeply, and one may generalize any of our bounded quantities to the general setting.  Further research may also investigate higher analogues of angles like three-dimensional solid angles.
    \item We bounded $R_{\text{gen}}(n)$, the size of the largest distinct-angle subset of an $n$ point configuration.  Alternatively, by viewing the point configuration as a complete graph on $n$ vertices, we may define $R_{\text{gen}}(n)$ as the number of vertices in the largest complete distinct-angle sub-graph.  Instead of removing vertices, one might ask about removing edges until all angles left are distinct.
\end{enumerate}


\end{document}